\documentclass{amsart}
\usepackage{amsfonts,amscd,amsthm,amsgen,amsmath,amssymb}
\usepackage[all]{xy}
\usepackage[vcentermath]{youngtab}
\newtheorem{theorem}{Theorem}[section]

\theoremstyle{definition}
\newtheorem{definition}[theorem]{Definition}

\theoremstyle{remark}
\newtheorem{remark}[theorem]{Remark}

\numberwithin{equation}{section}

\begin{document}

\title{A Coboundary Morphism For The Grothendieck Spectral Sequence}
\author{David Baraglia}

\address{Mathematical sciences institute, The Australian National University, Canberra ACT 0200, Australia}


\email{david.baraglia@anu.edu.au}

\thanks{This work is supported by the Australian Research Council Discovery Project DP110103745.}


\date{\today}


\begin{abstract}
Given an abelian category $\mathcal{A}$ with enough injectives we show that a short exact sequence of chain complexes of objects in $\mathcal{A}$ gives rise to a short exact sequence of Cartan-Eilenberg resolutions. Using this we construct coboundary morphisms between Grothendieck spectral sequences associated to objects in a short exact sequence. We show that the coboundary preserves the filtrations associated with the spectral sequences and give an application of these result to filtrations in sheaf cohomology.

\end{abstract}

\maketitle


\section{Introduction}

Whenever spectral sequences occur in practice there is usually a corresponding filtration on the objects one would like to compute. The associated graded objects of the filtration coincide with the limit of the spectral sequence. There are also situations where such filtrations are of interest independent from the spectral sequence. For example this is true of the filtration associated to the Leray spectral sequence. In this paper we establish a general result allowing us to compare the filtrations associated to certain spectral sequences. Specifically we are concerned with the behavior of the Grothendieck spectral sequence on short exact sequences and the implications this has on the associated filtrations.\\

Let us recall the content of the Grothendieck spectral sequence. Let $\mathcal{A},\mathcal{B},\mathcal{C}$ be abelian categories, $F : \mathcal{A} \to \mathcal{B}$, $G : \mathcal{B} \to \mathcal{C}$ left exact functors. Suppose $\mathcal{A},\mathcal{B}$ have enough injectives and $F$ sends injective objects to $G$-acyclic objects. Given an object $A \in \mathcal{A}$ there is a spectral sequence $\{ E_r^{p,q}(A) , d_r \}$ consisting of objects in $\mathcal{C}$ and filtration $F^p R^n(G \circ F)(A)$ on $R^n(G \circ F)$ such that the spectral sequence converges to the associated graded objects and such that $E_2^{p,q}(A) = R^p G( R^q F (A))$. The details of the spectral sequence and filtration along with our notation are described in Section \ref{gss}.

Now it is clear from the construction of the Grothendieck spectral sequence that given two objects $A,B \in \mathcal{A}$ and a morphism $A \to B$ the induced maps $R^n(G \circ F)(A) \to R^n(G \circ F)(B)$ respect the filtrations  and the maps induced on the associated graded objects come from a morphism $E_r^{p,q}(A) \to E_r^{p,q}(B)$ of spectral sequences. The maps $E_2^{p,q}(A) \to E_2^{p,q}(B)$ are of course just the induced morphisms $R^p G( R^q F(A)) \to R^pG (R^q F(B))$. We thus have a good understanding of how the spectral sequences of $A$ and $B$ are related.

Consider now a short exact sequence $0 \to A \to B \to C \to 0$ in $\mathcal{A}$. The main question we seek to answer is how the spectral sequences and filtrations associated to $A$ and $C$ are related. Of course we get corresponding morphisms $E_r^{p,q}(A) \to E_r^{p,q}(B)$ and $E_r^{p,q}(B) \to E_r^{p,q}(C)$ but the composition $E_r^{p,q}(A) \to E_r^{p,q}(B) \to E_r^{p,q}(C)$ is trivial, so this does not help to directly relate $E_r^{p,q}(A)$ and $E_r^{p,q}(C)$. Our main result, Theorem \ref{main} is that there is a kind of coboundary morphism of spectral sequences $E_r^{p,q}(C) \to E_r^{p,q+1}(A)$ and closely related to this is the fact that the coboundary maps $R^n(G \circ F)(C) \to R^{n+1}(G \circ F)(A)$ in the long exact sequence associated to $0 \to A \to B \to C \to 0$ and $G \circ F$ respects the filtrations. We state the full result:
\begin{theorem}
Let $0 \to A \to B \to C \to 0$ be a short exact sequence in $\mathcal{A}$. There are morphisms $\delta_r : E_r^{p,q}(C) \to E_r^{p,q+1}(A)$ for $r \ge 2$ between the Grothendieck spectral sequences for $C$ and $A$ with the following properties:
\begin{itemize}
\item{$\delta_r$ commutes with the differentials $d_r$ and the induced map at the $(r+1)$-stage is $\delta_{r+1}$.}
\item{$\delta_2 : R^pG( R^q F(C)) \to R^pG( R^{q+1}F(A))$ is the map induced by the boundary morphism $R^qF(C) \to R^{q+1}F(A)$ in the long exact sequence of derived functors of $F$ associated to $0 \to A \to B \to C \to 0$.}
\item{The boundaries $R^n(G \circ F)(C) \to R^{n+1}(G \circ F)(A)$ for the long exact sequence associated to $G \circ F$ send $F^p R^n(G \circ F)(C)$ to $F^p R^{n+1}(G \circ F)(A)$ and thus induce maps $E_\infty^{p,q}(C) \to E_\infty^{p,q+1}(A)$. These maps coincide with $\delta_\infty$ where $\delta_\infty$ denotes the limit of the $\delta_r$.}
\end{itemize}
\end{theorem}

In Section \ref{app} we specialize to the case of the Leray spectral sequence. Let $X,Y$ be paracompact spaces and $f : X \to Y$ continuous. Consider on $X$ the exponential sequence $0 \to \mathbb{Z} \to \underline{\mathbb{C}} \to \underline{\mathbb{C}}^* \to 0$, where $\underline{A}$ indicates the sheaf of continuous functions valued in an abelian group $A$. A well known fact is that the coboundary map $\delta : H^n(X,\underline{\mathbb{C}}^*) \to H^{n+1}(X,\mathbb{Z})$ is an isomorphism for $n \ge 1$. Note however that the Leray spectral sequences associated to the map $f : X \to Y$ determine filtrations on $H^n(X,\underline{\mathbb{C}}^*)$ and $H^{n+1}(X,\mathbb{Z})$. It is natural then to ask how these two filtrations compare under the coboundary $\delta$. We give the answer in Theorem \ref{cz}, as an application of Theorem \ref{main}.\\

Section \ref{cer} contains the main technical result needed for Theorem \ref{main}. The result here is that given a short exact sequence of chain complexes $0 \to A^* \to B^* \to C^* \to 0$ in an abelian category with enough injectives one can construct Cartan-Eilenberg resolutions for $A^*,B^*,C^*$ in such a way that they fit into a short exact sequence of double complexes. Section \ref{gss} recalls the necessary details of the Grothendieck spectral sequence, Section \ref{bses} is the proof of the main result and Section \ref{app} the application to sheaf cohomology previously described.


\section{Cartan-Eilenberg resolutions}\label{cer}

Let $\mathcal{A}$ be an abelian category. We recall the notion of a Cartan-Eilenberg resolution.
\begin{definition} Let $A^*$ be a complex in $\mathcal{A}$. A {\em Cartan-Eilenberg resolution} \cite{gm},\cite{ce} of $A^*$ is a sequence of complexes $I^{0,*},I^{1,*},I^{2,*},\dots$ together with chain maps $I^{p,*} \to I^{p+1,*}$ and injective chain map $A^* \to I^{0,*}$ such that
\begin{itemize}
\item{each object $I^{p,q}$ is injective,}
\item{the sequence $0 \to A^* \to I^{0,*} \to I^{1,*} \to \cdots$ is exact,}
\item{let $Z^q(A)$ denote the degree $q$ cocycles of $A^*$ and $Z^{p,q}(I)$ the degree $q$ cocycles of $I^{p,*}$. The induced sequence on cocycles $Z^q(A) \to Z^{0,q}(I) \to Z^{1,q}(I) \to \cdots$ is required to be an injective resolution of $Z^q(A)$. In particular each $Z^{p,q}(I)$ is injective,}
\item{similarly for coboundaries and cohomology the induced sequences are injective resolutions.}
\end{itemize}
\end{definition}

Our main result is that for a short exact sequence of chain complexes in an abelian category with enough injectives one can construct corresponding short exact sequence of Cartan-Eilenberg resolutions.

\begin{theorem}\label{short} Let $0 \to A^* \to B^* \to C^* \to 0$ be a short exact sequence of chain complexes in an abelian category with enough injectives. Then there exists a short exact sequence $0 \to I^* \to J^* \to K^* \to 0$ of chain complexes of injectives, together with injections $A^q \to I^q$, $B^q \to J^q$, $C^q \to K^q$ for all $q$ forming a commutative diagram
\begin{equation*}\xymatrix{
0 \ar[r] & A^* \ar[r] \ar[d] & B^* \ar[r] \ar[d] & C^* \ar[r] \ar[d] & 0 \\
0 \ar[r] & I^* \ar[r] & J^* \ar[r] & K^* \ar[r] & 0
}
\end{equation*}
Let $Z^*(A),B^*(A),H^*(A)$ denote the cocycles, coboundaries and cohomology for the complex $A^*$ and use similar notation for the other complexes. We may in addition make the choices so that the $Z^*(I),B^*(I),H^*(I)$ are all injectives, the natural maps $Z^*(A) \to Z^*(I)$, $B^*(A) \to B^*(I)$, $H^*(A) \to H^*(I)$ are injective and such that the corresponding statements for $J^*,K^*$ hold as well. Moreover if $A^q = B^q = C^q = 0$ whenever $q < 0$ then we may choose the $I^*,J^*,K^*$ so that likewise $I^q = J^q = K^q = 0$ for $q < 0$.
\end{theorem}

Before we prove Theorem \ref{short} let us state the main result and show how it follows:
\begin{theorem}\label{ce}Let $0 \to A^* \to B^* \to C^* \to 0$ be a short exact sequence of chain complexes in an abelian category with enough injectives. Then there exists Cartan-Eilenberg resolutions $0 \to A^* \to I^{0,*} \to I^{1,*} \to \cdots$, $0 \to B^* \to J^{0,*} \to J^{1,*} \to \cdots$, $0 \to C^* \to K^{0,*} \to K^{1,*} \to \cdots$ for $A^*,B^*,C^*$ and maps $I^{q,*} \to J^{q,*} \to K^{q,*}$ forming a commutative diagram of chain complexes
\begin{equation*}\xymatrix{
& 0 \ar[d] & 0 \ar[d] & 0 \ar[d] & \\
0 \ar[r] & A^* \ar[r] \ar[d] & B^* \ar[r] \ar[d] & C^* \ar[r] \ar[d] & 0 \\
0 \ar[r] & I^{0,*} \ar[r] \ar[d] & J^{0,*} \ar[r] \ar[d] & K^{0,*} \ar[r] \ar[d] & 0 \\
0 \ar[r] & I^{1,*} \ar[r] \ar[d] & J^{1,*} \ar[r] \ar[d] & K^{1,*} \ar[r] \ar[d] & 0 \\
& \vdots & \vdots & \vdots &
}
\end{equation*}
where in addition the rows are exact. Moreover if $A^q = B^q = C^q = 0$ for $q < 0$ we may choose the chain complex so that in addition $I^{p,q} = J^{p,q} = K^{p,q} = 0$ for $q < 0$ as well.
\end{theorem}
\begin{proof}
We use Theorem \ref{short} to construct the first row $0 \to I^{0,*} \to J^{0,*} \to K^{0,*} \to 0$ forming a commutative diagram with exact rows and columns
\begin{equation*}\xymatrix{
& 0 \ar[d] & 0 \ar[d] & 0 \ar[d] & \\
0 \ar[r] & A^* \ar[r] \ar[d]^{\alpha^*} & B^* \ar[r] \ar[d]^{\beta^*} & C^* \ar[r] \ar[d]^{\gamma^*} & 0 \\
0 \ar[r] & I^{0,*} \ar[r] & J^{0,*} \ar[r] & K^{0,*} \ar[r] & 0
}
\end{equation*}
Next take the cokernels of the vertical maps to obtain
\begin{equation*}\xymatrix{
& 0 \ar[d] & 0 \ar[d] & 0 \ar[d] & \\
0 \ar[r] & A^* \ar[r] \ar[d]^{\alpha^*} & B^* \ar[r] \ar[d]^{\beta^*} & C^* \ar[r] \ar[d]^{\gamma^*} & 0 \\
0 \ar[r] & I^{0,*} \ar[r] \ar[d] & J^{0,*} \ar[r] \ar[d] & K^{0,*} \ar[r] \ar[d] & 0 \\
0 \ar[r] & coker(\alpha^*) \ar[r] \ar[d] & coker(\beta^*) \ar[r] \ar[d] & coker(\gamma^*) \ar[r] \ar[d] & 0 \\
& 0 & 0 & 0 &
}
\end{equation*}
Note that the bottom row is exact by the Nine lemma \cite{fre}. We may now apply Theorem \ref{short} to $0 \to coker(\alpha^*) \to coker(\beta^*) \to coker(\gamma^*) \to 0$ to obtain the next row $0 \to I^{1,*} \to J^{1,*} \to K^{1,*} \to 0$. Continuing in this fashion we construct the desired double complex of chain complexes with exact columns and short exact rows. It remains to show that the resolutions so obtained of $A^*,B^*,C^*$ are indeed Cartan-Eilenberg resolutions. We will show that $0 \to A^* \to I^{0,*} \to I^{1,*} \to \cdots$ is a Cartan-Eilenberg resolution. The proofs for $B^*$ and $C^*$ are identical. By construction we have that the $Z^{p,*}(I),B^{p,*}(I),H^{p,*}(I)$ are all injective objects and that the maps $Z^*(A) \to Z^{0,*}(I)$, $B^*(A) \to B^{0,*}(I)$, $H^*(A) \to H^{0,*}(I)$ are all injective. It remains to that the sequences $\cdots \to Z^{p,*}(I) \to Z^{p+1,*}(I) \to \cdots $, $\cdots \to B^{p,*}(I) \to B^{p+1,*}(I) \to \cdots $, $\cdots \to H^{p,*}(I) \to H^{p+1,*}(I) \to \cdots $ are exact.\\

Consider the exact sequence $\cdots \to I^{p-1,*} \to I^{p,*} \to I^{p+1,*} \to \cdots$ of chain complexes. Letting $C^{p,*}$ denote the cokernel of $I^{p-1,*} \to I^{p,*}$ we get a commutative diagram
\begin{equation*}\xymatrix{
& & & 0 \ar[dr] & & 0 & \\
\ar[dr] & & & & C^{p,*} \ar[dr] \ar[ur] & & \\
\dots \ar[r] & I^{p-1,*} \ar[dr] \ar[rr] & & I^{p,*} \ar[ur] \ar[rr] & & I^{p+1,*} \ar[r] \ar[dr] & \dots \\
& & C^{p-1,*} \ar[dr] \ar[ur] & & & & \\
& 0 \ar[ur] & & 0 & & &
}
\end{equation*}
where the diagonal sequences are exact. Moreover by the construction of the complexes $I^{p,*}$ we know that the maps $C^{p,*} \to I^{p+1,*}$ are injective on cocycles, coboundaries and cohomology as in Theorem \ref{short}. Let $Z^{p,*}(C),B^{p,*}(C),H^{p,*}(C)$ denote the cocycles, coboundaries and cohomology for the complex $C^{p,*}$. Let us check exactness at $I^{p,*}$ for cocycles. The kernel of $Z^{p,*}(I) \to Z^{p+1,*}(I)$ is equal to the kernel of $Z^{p,*}(I) \to Z^{p,*}(C)$ which in turn is equal to $Z^{p-1,*}(C)$. Now using the fact that $H^{p-2,*}(C) \to H^{p-1,*}(I)$ is injective we immediately see that $Z^{p-1,*}(I) \to Z^{p-1,*}(C)$ is surjective. It follows that the kernel of $Z^{p,*}(I) \to Z^{p+1,*}(I)$ is exactly the image $Z^{p-1,*}(I) \to Z^{p,*}(I)$ as required. A similar proof shows exactness at $I^{p,*}$ for cohomology and boundaries.\\

The final statement about vanishing of negative degree terms follows from the similar result for Theorem \ref{short}.
\end{proof}

\begin{proof}[Proof of Theorem \ref{short}] Given a short exact sequence $0 \to A^* \to B^* \to C^* \to 0$ of chain complexes introduce the following objects: cocycles $Z^*(A),Z^*(B),Z^*(C)$, coboundaries $B^*(A),B^*(B),B^*(C)$, cohomologies $H^*(A),H^*(B),H^*(C)$ and two more sets of objects $W^*(A),W^*(B),W^*(C), X^*(A),X^*(B),X^*(C)$. We define $W^q(A)$ to be the kernel of the induced map $H^q(A) \to H^q(B)$ and $W^q(B),W^q(C)$ are similarly defined as kernels in the long exact sequence in cohomology. We define $X^q(A)$ to be the kernel of the composition $Z^q(A) \to H^q(A) \to H^q(B)$ and similarly define $X^q(B),X^q(C)$. The objects so defined fit in to a variety of short exact sequences as follows.

First the long exact sequence in cohomology we have:
\begin{eqnarray}
&0 \to W^q(A) \to H^q(A) \to W^q(B) \to 0 \label{es1}, \\
&0 \to W^q(B) \to H^q(B) \to W^q(C) \to 0 \label{es2}, \\
&0 \to W^q(C) \to H^q(C) \to W^{q+1}(A) \to 0 \label{es3}.
\end{eqnarray}
From the definition of the objects $X^q(A),X^q(B),X^q(C)$ we have:
\begin{eqnarray}
&0 \to X^q(A) \to Z^q(A) \to W^q(B) \to 0 \label{es4}, \\
&0 \to X^q(B) \to Z^q(B) \to W^q(C) \to 0 \label{es5}, \\
&0 \to X^q(C) \to Z^q(C) \to W^{q+1}(A) \to 0 \label{es6}.
\end{eqnarray}
From the definition of cohomology:
\begin{eqnarray}
&0 \to B^q(A) \to Z^q(A) \to H^q(A) \to 0 \label{es7}, \\
&0 \to B^q(B) \to Z^q(B) \to H^q(B) \to 0 \label{es8}, \\
&0 \to B^q(C) \to Z^q(C) \to H^q(C) \to 0 \label{es9}.
\end{eqnarray}
Also from the definition of the $X^q(A),X^q(B),X^q(C)$:
\begin{eqnarray}
&0 \to B^q(A) \to X^q(A) \to W^q(A) \to 0 \label{es10}, \\
&0 \to B^q(B) \to X^q(B) \to W^q(B) \to 0 \label{es11}, \\
&0 \to B^q(C) \to X^q(C) \to W^q(C) \to 0 \label{es12}.
\end{eqnarray}
From the definition of cycles and boundaries:
\begin{eqnarray}
&0 \to Z^q(A) \to A^q \to B^{q+1}(A) \to 0 \label{es13}, \\
&0 \to Z^q(B) \to B^q \to B^{q+1}(B) \to 0, \\
&0 \to Z^q(C) \to C^q \to B^{q+1}(C) \to 0 \label{es15}.
\end{eqnarray}
Finally a few more exact sequences that can be easily shown:
\begin{eqnarray}
&0 \to X^q(A) \to B^q(B) \to B^q(C) \to 0 \label{es16}, \\
&0 \to Z^q(A) \to Z^q(B) \to X^q(C) \to 0 \label{es17}, \\
&0 \to Z^q(A) \to X^q(B) \to B^q(C) \to 0 \label{es18}, \\
&0 \to A^q \to B^q \to C^q \to 0. \label{es19}
\end{eqnarray}
We choose five families of injective objects indexed by the integer $q$ which we suggestively denote as follows $W^q(I),W^q(J),W^q(K),B^q(I),B^q(K)$. By assumption $\mathcal{A}$ has enough injectives so we can choose these objects together with injections $W^q(A) \to W^q(I)$, $W^q(B) \to W^q(J)$, $W^q(C) \to W^q(K)$, $B^q(A) \to B^q(I)$, $B^q(C) \to B^q(K)$. We aim ultimately to construct chain complexes $I^*,J^*,K^*$ such that the objects $W^q(I),W^q(J),W^q(K),B^q(I),B^q(K)$ agree with the objects that their notation suggests.

We further define the following objects
\begin{eqnarray}
H^q(I) &=& W^q(I) \oplus W^q(J), \label{hi} \\
H^q(J) &=& W^q(J) \oplus W^q(K), \\
H^q(K) &=& W^q(K) \oplus W^{q+1}(I), \\
B^q(J) &=& W^q(I) \oplus B^q(I) \oplus B^q(K), \\
X^q(I) &=& W^q(I) \oplus B^q(I),  \\
X^q(J) &=& W^q(I) \oplus W^q(J) \oplus B^q(I) \oplus B^q(K), \label{xk} \\
X^q(K) &=& W^q(K) \oplus B^q(K), \\
Z^q(I) &=& W^q(I) \oplus W^q(J) \oplus B^q(I), \label{zi} \\
Z^q(J) &=& W^q(I) \oplus W^q(J) \oplus W^q(K) \oplus B^q(I) \oplus B^q(K), \\
Z^q(K) &=& W^q(K) \oplus W^{q+1}(I) \oplus B^q(K), \label{zk} \\
I^q &=& Z^q(I) \oplus B^{q+1}(I), \label{defi} \\
J^q &=& Z^q(J) \oplus B^{q+1}(J), \label{defj} \\
K^q &=& Z^q(K) \oplus B^{q+1}(K). \label{defk}
\end{eqnarray}
The idea behind these definitions is that assuming the existence of the desired complexes $I^*,J^*,K^*$ we have exact sequences like Equations (\ref{es1})-(\ref{es19}). If in addition all the cocycles, coboundaries and so forth are injective then all these sequences would be split and the above definitions would hold. In fact we will now give $I^*,J^*,K^*$ as defined in (\ref{defi})-(\ref{defk}) the structure of cochain complexes so that these assumptions actually hold true. First of all note that using inclusions and projections there are unique ways to define maps between the various objects so that the short exact sequences corresponding to (\ref{es1})-(\ref{es19}) hold (with $I,J,K$ in place of $A,B,C$). For example the exact sequence $0 \to W^q(I) \to H^q(I) \to W^q(J) \to 0$ is just the spit exact sequence $0 \to W^q(I) \to W^q(I) \oplus W^q(J) \to W^q(J) \to 0$. Next we define a differential $d : I^q \to I^{q+1}$ as the composition $I^q \to B^{q+1}(I) \to Z^{q+1}(I) \to I^{q+1}$. Note that the three maps in this composition correspond to maps in the sequences (\ref{es13}),(\ref{es7}) and (\ref{es13}) again. Thus by (\ref{es13}) we see that $d^2 = 0$ so that this indeed defines a differential. We can similarly define differentials for $J^*$ and $K^*$.

We claim that the maps $I^q \to J^q \to K^q$ as in (\ref{es19}) are in fact morphisms of complexes. To see this one easily checks that the following diagram commutes:
\begin{equation*}\xymatrix{
I^q \ar[r] \ar[dd]^{(\ref{es19})} & B^{q+1}(I) \ar[r] \ar[d]^{(\ref{es10})} & Z^{q+1}(I) \ar[r] \ar[dd]^{(\ref{es17})} & I^{q+1} \ar[dd]^{(\ref{es19})} \\
& X^{q+1}(I) \ar[d]^{(\ref{es16})} & & \\
J^q \ar[r] \ar[dd]^{(\ref{es19})} & B^{q+1}(J) \ar[r] \ar[dd]^{(\ref{es16})} & Z^{q+1}(J) \ar[r] \ar[d]^{(\ref{es17})} & J^{q+1} \ar[dd]^{(\ref{es19})} \\
& & X^{q+1}(K) \ar[d]^{(\ref{es6})} & \\
K^q \ar[r] & B^{q+1}(K) \ar[r] & Z^{q+1}(K) \ar[r] & K^{q+1}
}
\end{equation*}
The labels on the vertical arrows indicate that these maps are the same as the maps defined as in the exact sequence indicated by the label. The horizontal rows are precisely the maps defining the differentials for $I^*,J^*,K^*$ so commutativity of this diagram implies the maps $I^* \to J^* \to K^*$ are chain maps.\\

Now that we have a short exact sequence $0 \to I^* \to J^* \to K^* \to 0$ of chain complexes we get associated spaces of cocycles, coboundaries, cohomology and so on. One can check easily that the $Z^*(I),Z^*(J),Z^*(K)$ are indeed the cocycles, $B^*(I),B^*(J),B^*(K)$ are indeed the coboundaries and so forth. Thus the objects $W^*(I),W^*(J),W^*(K),B^*(I),B^*(J)$ and the various objects defined in Equations (\ref{hi})-(\ref{zk}) coincide with their namesake.\\

The next part of the proof is to construct chain maps $A^* \to I^*$, $B^* \to J^*$ and $C^* \to K^*$. Such maps would induce maps between cocycles, coboundaries and so on. Thus our strategy will be to construct these maps by working backwards, starting from the existing maps $W^q(A) \to W^q(I)$, $W^q(B) \to W^q(J), \dots $ and working our way through the exact sequences (\ref{es1})-(\ref{es19}). In other words for each of the short exact sequences (\ref{es1})-(\ref{es19}), we want to construct a corresponding commutative diagram of short exact sequences. For example we start off with (\ref{es1}). The desired commutative diagram is as follows:
\begin{equation*}\xymatrix{
0 \ar[r] & W^q(I) \ar[r] & H^q(I) \ar[r] & W^q(J) \ar[r] & 0 \\
0 \ar[r] & W^q(A) \ar[r] \ar[u] & H^q(A) \ar[r] \ar[u] & W^q(B) \ar[r] \ar[u] & 0
}
\end{equation*}
where the maps $W^q(A) \to W^q(I)$, $W^q(B) \to W^q(J)$ are the previously chosen injections. Since $H^q(I) = W^q(I) \oplus W^q(J)$ we need to chose maps $f : H^q(A) \to W^q(I)$ and $g : H^q(A) \to W^q(J)$. The diagram will commute if and only if the following diagrams commute:
\begin{equation*}\xymatrix{
& W^q(I) \\
W^q(A) \ar[ur] \ar[r] & H^q(A) \ar[u]^f
}
\end{equation*}
and
\begin{equation*}\xymatrix{
& W^q(J) \\
H^q(A) \ar[ur]^g \ar[r] & W^q(B) \ar[u]
}
\end{equation*}
In the first case such a map $f$ exists because $W^q(I)$ is an injective object and $W^q(A) \to H^q(A)$ is an injection. In the second case $g$ exists just by defining it to be the composition. Note also that the map $H^q(A) \to H^q(I)$ so defined is injective.\\

Using the exact same reasoning we construct similar injections $H^q(B) \to H^q(J)$, $H^q(C) \to H^q(K)$ yielding commutative diagrams corresponding to the sequences (\ref{es2}) and (\ref{es3}). Similarly we construct injections $X^k(A) \to X^k(I)$ and $X^k(C) \to X^k(K)$ yielding commutative diagrams corresponding to (\ref{es10}) and (\ref{es12}). From this we may construct an injection $B^q(B) \to B^q(J)$ yielding commutative diagram corresponding to (\ref{es16}). However when we consider the construction of a map $Z^q(A) \to Z^q(I)$ we run into a complication, namely there are two exact sequences (\ref{es4}),(\ref{es7}) with $Z^q(A)$ as the middle term. We would like to choose the map $Z^q(A) \to Z^q(I)$ to yield commutative diagrams corresponding to both of these exact sequences. For clarity we write out the two desired commutative diagrams
\begin{equation*}\xymatrix{
0 \ar[r] & X^q(I) \ar[r] & Z^q(I) \ar[r] & W^q(J) \ar[r] & 0 \\
0 \ar[r] & X^q(A) \ar[r] \ar[u] & Z^q(A) \ar[r] \ar[u] & W^q(B) \ar[r] \ar[u] & 0
}
\end{equation*}
and
\begin{equation*}\xymatrix{
0 \ar[r] & B^q(I) \ar[r] & Z^q(I) \ar[r] & H^q(I) \ar[r] & 0 \\
0 \ar[r] & B^q(A) \ar[r] \ar[u] & Z^q(A) \ar[r] \ar[u] & H^q(A) \ar[r] \ar[u] & 0
}
\end{equation*}
All the maps in these diagrams are already determined except the desired map $Z^q(A) \to Z^q(I)$. By sequence (\ref{zi}) we have $Z^q(I) = W^q(I) \oplus W^q(J) \oplus B^q(I)$, so we need to choose maps $Z^q(A) \to W^q(I)$, $Z^q(A) \to W^q(J)$, $Z^q(A) \to B^q(I)$. The components $Z^q(A) \to W^q(I)$ and $Z^q(A) \to W^q(J)$ are determined by requiring a commutative diagram
\begin{equation*}\xymatrix{
& H^q(I) \\
Z^q(A) \ar[ur] \ar[r] & H^q(A) \ar[u]
}
\end{equation*}
in addition the commutative diagram
\begin{equation*}\xymatrix{
H^q(I) \ar[r] & W^q(J) \\
H^q(A) \ar[u] \ar[r] & W^q(B) \ar[u]
}
\end{equation*}
which follows from the construction of the map $H^q(A) \to H^q(I)$ ensures that we also have a commutative diagram
\begin{equation*}\xymatrix{
& W^q(J) \\
Z^q(A) \ar[ur] \ar[r] & W^q(B) \ar[u]
}
\end{equation*}
Next we need a map $Z^q(A) \to B^q(I)$. Let us use injectivity of $B^q(I)$ to choose such a map so that the following commutes
\begin{equation*}\xymatrix{
& B^q(I) \\
X^q(A) \ar[ur] \ar[r] & Z^q(A) \ar[u]
}
\end{equation*}
where $X^q(A) \to B^q(I)$ is the composition of the map $X^q(A) \to X^q(I)$ and the projection $X^q(I) \to B^q(I)$. Therefore we have a commutative diagram
\begin{equation*}\xymatrix{
& & B^q(I) \\
B^q(A) \ar[r] \ar[urr] & X^q(A) \ar[ur] \ar[r] & Z^q(A) \ar[u]
}
\end{equation*}
Finally we need to check that the composition $X^q(A) \to X^q(I) \to W^q(I)$ agrees with $X^q(A) \to Z^q(A) \to H^q(A) \to H^q(I) \to W^q(I)$, where the last map $H^q(I) \to W^q(I)$ is the projection. But this is straightforward since the composition $X^q(A) \to Z^q(A) \to H^q(A)$ equals the composition $X^q(A) \to W^q(A) \to H^q(A)$. From this is follows that the above two diagrams corresponding to sequences (\ref{es4}) and (\ref{es7}) commute.\\

By the exact same argument we construct an injection $Z^q(C) \to Z^q(K)$ yielding two commutative diagrams corresponding to (\ref{es6}) and (\ref{es9}). The next injection to construct is $X^q(B) \to X^q(J)$. Again there is a complication since we want to choose the map to yield two commutative diagrams corresponding to (\ref{es11}) and (\ref{es18}). From sequence (\ref{xk}) we have $X^q(J) = W^q(I) \oplus W^q(J) \oplus B^q(I) \oplus B^q(K)$ so we need to define maps $X^q(B) \to W^q(I)$, $X^q(B) \to W^q(J)$, $X^q(B) \to B^q(I)$ and $X^q(B) \to B^q(K)$. Of these maps the ones into $W^q(J)$ and $B^q(K)$ are already determined by commutativity. Since $W^q(I) \oplus B^q(I) = X^q(I)$ the remaining two terms can be expressed as a map $X^q(B) \to X^q(I)$. For commutativity we need that this map fits into a commutative diagram as follows:
\begin{equation*}\xymatrix{
& X^q(I) & \\
Z^q(A) \ar[ur]^a \ar[r] & X^q(B) \ar[u] & B^q(B) \ar[l] \ar[ul]_b
}
\end{equation*}
where the map $a : Z^q(A) \to X^q(I)$ is the composition of the map $Z^q(A) \to Z^q(I)$ and the projection $Z^q(I) \to X^q(I)$ and the map $b : B^q(B) \to X^q(I)$ is defined in a similar manner. To proceed let $i_1 : Z^q(A) \to X^q(B)$ and $i_2 : B^q(B) \to X^q(B)$ be the inclusions. We observe that the kernel of $(i_1,0)+(0,i_2) : Z^q(A) \oplus B^q(B) \to X^q(B)$ is precisely $(j_1,-j_2) : X^q(A) \to Z^q(A) \oplus B^q(B)$ where $j_1 : X^q(A) \to Z^q(A)$, $j_2 : X^q(A) \to B^q(B)$ are the inclusions. Next we observe that $a \circ j_1 = b \circ j_2$ since both maps are just the map $X^q(A) \to X^q(I)$ we have previously constructed. It follows that the map $(a,0)+(0,b) : Z^q(A) \oplus B^q(B) \to X^q(I)$ factors to a map $Q \to X^q(I)$ where $Q$ is the cokernel of $(j_1,-j_2) : X^q(A) \to Z^q(A) \oplus B^q(B)$. Obviously the map $(i_1,0)+(0,i_2) : Z^q(A) \oplus B^q(B) \to X^q(B)$ factors to an injection $Q \to X^q(B)$. Now since $X^q(I)$ is injective there exists a map $X^q(B) \to X^q(I)$ yielding a commutative diagram
\begin{equation*}\xymatrix{
& X^q(I) \\
Q \ar[r] \ar[ur] & X^q(B) \ar[u]
}
\end{equation*}
Moreover it follows easily that we have a commutative diagram as follows
\begin{equation*}\xymatrix{
& & X^q(I) \\
Z^q(A) \ar[r] \ar[urr]^a & Q \ar[r] \ar[ur] & X^q(B) \ar[u]
}
\end{equation*}
and similarly with $B^q(B)$ in place of $Z^q(A)$. Thus the map $X^q(B) \to X^q(I)$ has the desired properties. From this it follows easily that the corresponding map $X^q(B) \to X^q(J)$ we have now constructed yields commutative diagrams corresponding to sequences (\ref{es11}) and (\ref{es18}).\\

Along the same lines as has been described so far one can find an injective map $Z^q(B) \to Z^q(J)$ yielding commutative diagrams corresponding to (\ref{es5}),(\ref{es8}) and (\ref{es17}). Next one constructs injections $A^q \to I^q$ and $B^q \to K^q$ yielding commutative diagrams corresponding to (\ref{es13}) and (\ref{es15}). Finally one constructs an injection $B^q \to J^q$. Once all of this is done we have commutative diagrams corresponding to Equations (\ref{es1})-(\ref{es19}). From the construction of the differentials on $I^*,J^*$ and $K^*$ we see that the maps $A^* \to I^*$, $B^* \to J^*$ and $C^* \to K^*$ are chain maps. The commutative diagram corresponding to (\ref{es19}) is
\begin{equation*}\xymatrix{
0 \ar[r] & I^q \ar[r]  & J^q \ar[r]  & K^q \ar[r] & 0 \\
0 \ar[r] & A^q \ar[r] \ar[u] & B^q \ar[r] \ar[u] & C^q \ar[r] \ar[u] & 0
}
\end{equation*}
Finally by construction the induced maps $Z^q(A) \to Z^q(I)$, $Z^q(B) \to Z^q(J)$, $Z^q(C) \to Z^q(K)$ on cocycles are injective and similarly for coboundaries and cohomologies.\\

For the final statement about vanishing in negative degrees suppose $A^q = B^q = C^q = 0$ for $q<0$. It follows easily that $W^q(A) = W^q(B) = W^q(C) = B^q(A) = B^q(C)$ whenever $q < 0$. Therefore in the above constructions we may choose $W^q(I) = W^q(J) = W^q(K) = B^q(I) = B^q(K) = 0$ for $q < 0$ and it follows directly that $I^q = J^q = K^q = 0$ when $q < 0$.
\end{proof}


\section{The Grothendieck spectral sequence}\label{gss}

We recall the construction of the Grothendieck spectral sequence and establish some of its basic properties. Let $\mathcal{A},\mathcal{B},\mathcal{C}$ be abelian categories, $F : \mathcal{A} \to \mathcal{B}$, $G : \mathcal{B} \to \mathcal{C}$ left exact functors. Suppose $\mathcal{A},\mathcal{B}$ have enough injectives and $F$ sends injective objects to $G$-acyclic objects. Let $A$ be an object of $\mathcal{A}$. Let $M^*, A \to M^0$ be an injective resolution of $A$ and set $A^* = F(M^*)$. Let $0 \to A^* \to I^{0,*} \to I^{1,*} \to \cdots$ be a Cartan-Eilenberg resolution of $A^*$. Since $\mathcal{B}$ has enough injectives it is well known that a Cartan-Eilenberg resolution exists, but we can also deduce this from Theorem \ref{ce} using the complex $0 \to A^* \to A^* \to 0$. Now set $R^{p,q} = G(I^{p,q})$ to obtain a double complex. Note that using Theorem \ref{ce} we can also assume that $I^{p,q} = 0$ if $p<0$ or $q<0$ and so the same is true of $R^{p,q}$. As usual for a double complex we have two natural filtrations and thus two spectral sequences associated to the double complex $R^{p,q}$, both abutting to the cohomology of the associated single complex $R^*$. We consider these two spectral sequences in turn.\\ 

Consider first the spectral sequence corresponding to the filtration by $q$-degree, the terms of which we denote by $\tilde{E}_r^{p,q}$. By assumption each $A^p = F(M^p)$ is $G$-acyclic since $M^p$ is injective. The $\tilde{E}_1^{p,q}$ terms are obtained by taking cohomology of the double complex $R^{p,q}$ in the $p$ direction so we find
\begin{equation}\label{etild1}
\tilde{E}_1^{p,q} = \left\{ \begin{matrix} G(A^q) & p = 0,\\ 0 & p > 0. \end{matrix} \right.
\end{equation}
Note also that $G(A^q) = (G \circ F)(M^q)$ so that on passing to the next stage of the spectral sequence we have
\begin{equation*}
\tilde{E}_2^{p,q} = \tilde{E}_\infty^{p,q} = \left\{ \begin{matrix} R^q(G \circ F)(A) & p = 0,\\ 0 & p > 0. \end{matrix} \right.
\end{equation*}
We deduce that the degree $n$ cohomology of the single complex associated to $R^{p,q}$ is given by $R^n(G \circ F)(A)$.\\

Consider now the second spectral sequence associated to $R^{p,q}$ corresponding to the filtration $F^kR^n$ by $p$-degree where $F^k R^n = \bigoplus_{p,q | p \ge k} R^{p,q}$. We denote by $E_r^{p,q}$ the associated spectral sequence which as we now know abuts to $R^n(G \circ F)(A)$. To be more precise there is a filtration
\begin{eqnarray*}
&&0 = F^{n+1}R^n(G \circ F)(A) \subseteq F^{n}R^n(G \circ F)(A) \subseteq \cdots \\
&& \; \; \cdots \subseteq F^{1}R^n(G \circ F)(A) \subseteq F^{0}R^n(G \circ F)(A) = R^n(G \circ F)(A)
\end{eqnarray*}
of $R^n(G \circ F)(A)$ such that $E_\infty^{p,q} \simeq F^pR^{p+q}(G \circ F)(A)/F^{p+1}R^{p+q}(G \circ F)(A)$. In fact if we let $H^n(F^pR^*)$ denote the cohomology of $F^pR^*$ then $F^pR^n(G \circ F)(A)$ is the image of $H^n(F^pR^*)$ under the map $H^n(F^pR^*) \to H^n(R^*)$ induced by the inclusion $F^pR^* \to R^*$.\\

Getting back to the spectral sequence $E_r^{p,q}$, we first compute $E_1^{p,q}$ by taking the cohomology of the double complex $R^{p,q}$ in the $q$ direction. For this we make use of the fact that $R^{p,q} = G(I^{p,q})$ where $I^{p,q}$ is a Cartan-Eilenberg resolution of $A^q$. If $Z^{p,q},B^{p,q},H^{p,q}$ denote the cocycles, coboundaries and cohomologies of $I^{p,q}$ in the $q$ direction then since $I^{p,q}$ is a Cartan-Eilenberg resolution the $Z^{p,q},B^{p,q},H^{p,q}$ are all injective and so the sequences $0 \to B^{p,q} \to Z^{p,q} \to H^{p,q} \to 0$ and $0 \to Z^{p,q} \to I^{p,q} \to B^{p,q+1} \to 0$ are split. Applying $G$ it easily follows that the cohomology of $R^{p,q}$ in the $q$ direction is $G( H^{p,q})$ so that $E_1^{p,q} = G( H^{p,q} )$. Using the Cartan-Eilenberg property again we have that $0 \to R^qF(A) \to H^{0,q} \to H^{1,q} \to \cdots$ is an injective resolution of $R^qF(A)$ so that the $E_2$ stage of this spectral sequence is given by
\begin{equation*}
E_2^{p,q} = R^pG( R^qF(A)).
\end{equation*}
This is the {\em Grothendieck spectral sequence} \cite{ce},\cite{gm}. One can show that from the $E_2$ stage onwards the spectral sequence does not depend on the choice of Cartan-Eilenberg resolution \cite{voi2}.


\section{Behavior on short exact sequences}\label{bses}

For any object $A$ in $\mathcal{A}$ let $F^p R^n(G \circ F)(A)$ denote the filtration on $R^n(G \circ F)(A)$ corresponding to the Grothendieck spectral sequence $(E_r^{p,q}(A) , d_r )$. In particular $E_\infty^{p,q}(A) \simeq F^p R^n (G \circ F)(A) / F^{p+1} R^n (G \circ F)(A)$. Recall also that $E_2^{p,q}(A) = R^pG( R^q F (A))$.

\begin{theorem}\label{main}
Let $0 \to A \to B \to C \to 0$ be a short exact sequence in $\mathcal{A}$. There are morphisms $\delta_r : E_r^{p,q}(C) \to E_r^{p,q+1}(A)$ for $r \ge 2$ between the Grothendieck spectral sequences for $C$ and $A$ with the following properties:
\begin{itemize}
\item{$\delta_r$ commutes with the differentials $d_r$ and the induced map at the $(r+1)$-stage is $\delta_{r+1}$.}
\item{$\delta_2 : R^pG( R^q F(C)) \to R^pG( R^{q+1}F(A))$ is the map induced by the boundary morphism $R^qF(C) \to R^{q+1}F(A)$ in the long exact sequence of derived functors of $F$ associated to $0 \to A \to B \to C \to 0$.}
\item{The boundaries $R^n(G \circ F)(C) \to R^{n+1}(G \circ F)(A)$ for the long exact sequence associated to $G \circ F$ send $F^p R^n(G \circ F)(C)$ to $F^p R^{n+1}(G \circ F)(A)$ and thus induce maps $E_\infty^{p,q}(C) \to E_\infty^{p,q+1}(A)$. These maps coincide with $\delta_\infty$ where $\delta_\infty$ denotes the limit of the $\delta_r$.}
\end{itemize}
\end{theorem}
\begin{proof}
Let $M^*,N^*,P^*$ be injective resolutions of $A,B,C$ respectively. By the Horseshoe lemma \cite{wei} we can choose $M^*,N^*,P^*$ so that there exists a commutative diagram of objects in $\mathcal{A}$ where the columns are short exact sequences
\begin{equation*}\xymatrix{
0 \ar[r] &  A   \ar[d] \ar[r] &  M^0   \ar[r] \ar[d] & M^1   \ar[d] \ar[r] &  \cdots  \\
0 \ar[r] &  B  \ar[d] \ar[r] &  N^0  \ar[r] \ar[d] & N^1  \ar[d] \ar[r] &  \cdots  \\
0 \ar[r] &  C        \ar[r] &  P^0 \ar[r]        & P^1        \ar[r] &  \cdots
}
\end{equation*}
We thus have a short exact sequence $0 \to M^* \to N^* \to P^* \to 0$ of chain complexes. The sequence $0 \to F(M^*) \to F(N^*) \to F(P^*) \to 0$ is also exact since the $M^q$ are injective. Thus we may apply Theorem \ref{ce} to find Cartan-Eilenberg resolutions $0 \to F(M^*) \to I^{0,*} \to I^{1,*} \to \cdots$, $0 \to F(N^*) \to J^{0,*} \to J^{1,*} \to \cdots$, $0 \to F(M^*) \to K^{0,*} \to K^{1,*} \to \cdots$ with the properties described in Theorem \ref{ce}, in particular we can assume $I^{p,q} = J^{p,q} = K^{p,q} = 0$ if $p<0$ or $q < 0$.\\

Next we get double complexes $R^{p,q} = G(I^{p,q})$, $S^{p,q} = G(J^{p,q})$, $T^{p,q} = G(K^{p,q})$ with terms in $\mathcal{C}$ by applying the functor $G$. If we let $R^*,S^*,T^*$ denote the associated single complexes then we know that the degree $n$ cohomology $H^n(R^*)$ of $R^*$ coincides with $R^n(G \circ F)(A)$. Similarly for $S^*$ and $T^*$.\\

By assumption $F$ takes injective objects of $\mathcal{A}$ to $G$-acyclic objects in $\mathcal{B}$. It follows that the natural sequences $0 \to R^{p,q} \to S^{p,q} \to T^{p,q} \to 0$ and $0 \to R^n \to S^n \to T^n \to 0$ are exact. We claim that the long exact sequence $0 \to H^0(R^*) \to H^0(S^*) \to H^0(T^*) \to H^1(R^*) \to H^1(S^*) \to \dots$ coincides with the long exact sequence $0 \to (G \circ F)(A) \to (G \circ F)(B) \to (G \circ F)(C) \to R^1(G \circ F)(A) \to R^1(G \circ F)(B) \to \dots$. Indeed we will show that the map $(G \circ F)(M^*) \to R^*$ obtained from the composition $(G \circ F)(M^*) \to G( I^{0,*} ) \to R^*$ is a quasi-isomorphism, similarly for $B,C$. We thus have a commutative diagram of short exact sequences of complexes in $\mathcal{C}$
\begin{equation*}\xymatrix{
0 \ar[r] & (G \circ F)(M^*) \ar[r] \ar[d] & (G \circ F)(N^*) \ar[r] \ar[d] & (G \circ F)(P^*) \ar[r] \ar[d] & 0 \\
0 \ar[r] & R^* \ar[r]    & S^*  \ar[r]   & T^*   \ar[r] &    0
}
\end{equation*}
such that the vertical arrows are quasi-isomorphisms. The top row is exact since $F$ takes injectives to $G$-acyclic objects. We thus get a chain isomorphism between the corresponding long exact sequences. Note that the degree $n$ cohomology of $(G \circ F)(M^*)$ is precisely $R^n(G \circ F)(A)$ and similarly for $B,C$. Now to finish the claim we must show that $(G \circ F)(M^*) \to R^*$ is a quasi-isomorphism. For this we introduce filtrations $F'^k (G \circ F)(M^*) $ on $(G \circ F)(M^*)$ and $F'^k R^*$ on $R^*$. We set $F'^0 (G \circ F)(M^*) = (G \circ F)(M^*)$ and $F'^k (G \circ F)(M^*) = 0$ if $k > 0$. For $R^*$ we take $F'^k R^* = \bigoplus_{p,q | q \ge k} G(I^{p,q})$. The map $(G \circ F)(M^*) \to R^*$ is easily seen to preserve the filtrations so induces a map between the spectral sequences for these filtrations. One finds easily that the map at the $E_1$-stage is an isomorphism. Indeed $E_1^{p,q}$ for the filtration on $R^*$ is given by Equation (\ref{etild1}), where $A^q$ denotes $F(M^q)$. From this we indeed find that the map is an isomorphism at the $E_1$-stage. This is enough to show that $(G \circ F)(M^*) \to R^*$ is a quasi-isomorphism.\\

Introduce a filtration $F^k R^*$ on $R^*$ by setting $F^k S^* = \bigoplus_{p,q | p \ge k} R^{p,q}$. This filtration yields a corresponding filtration $F^k H^n(R^*)$ on the cohomology $H^n(R^*)$ by letting $F^p H^n(R^*)$ be the image in $H^n(R^*)$ of the degree $n$ cohomology of $F^p R^*$ under the natural inclusion $F^p R^* \to R^*$. We have a spectral sequence $(E_r^{p,q}(R^*) , d_r )$ corresponding to the filtration so $E_\infty^{p,q}(R^*) \simeq F^pS^{p+q} / F^{p+1}S^{p+q}$. Similarly for $S^*,T^*$ we have filtrations and spectral sequences defined in the same manner. In fact since $I^{p,q},J^{p,q},K^{p,q}$ are Cartan-Eilenberg resolutions we have as shown in Section \ref{gss} that the resulting spectral sequences $(E_r^{p,q}(R^*),d_r)$, $(E_r^{p,q}(S^*),d_r)$, $(E_r^{p,q}(T^*),d_r)$ are the Grothendieck spectral sequences corresponding to $A,B,C$. We thus have $E_2^{p,q}(R^*) = R^pG(R^qF(A))$, $E_2^{p,q}(S^*) = R^pG(R^qF(B))$ and $E_2^{p,q}(T^*) = R^pG(R^qF(C))$.\\

The spectral sequences for $R^*$ and $T^*$ are determined by corresponding exact couples $(A_1(R^*) , E_1(R^*))$, $(A_1(T^*) , E_1(T^*))$ where
\begin{eqnarray*}
A^{p,q}_1(R^*) &=& \bigoplus_{p,q} H^{p+q}( F^p R^* ) \\
E^{p,q}_1(R^*) &=& \bigoplus_{p,q} H^{p+q}( F^p R^* / F^{p+1} R^* )
\end{eqnarray*}
and similarly for $T^*$. To define the exact couple $(A_1(R^*),E_1(R^*))$ we must also give maps $i : A_1(R^*) \to A_1(R^*)$, $j : A_1(R^*) \to E_1(R^*)$ and $k : E_1(R^*) \to A_1(R^*)$. We take $i : A^{p,q}_1(R^*) \to A_1^{p-1,q+1}(R^*)$ to be the map in cohomology induced by the inclusions $ F^{p}R^* \to F^{p-1} R^*$, $j : A_1^{p,q}(R^*) \to E_1^{p,q}(R^*)$ induced by the projection $F^p R^* \to F^p R^* / F^{p+1} R^*$ and $k : E_1^{p,q}(R^*) \to A_1^{p+1,q}(R^*)$ to be the coboundary in the long exact sequence for $0 \to F^{p+1}R^* \to F^pR^* \to F^pR^* / F^{p+1} R^* \to 0$. Define similar maps $i,j,k$ in the case of $T^*$.\\

We define a map $\delta : (A_1(T^*) , E_1(T^*)) \to (A_1(R^*) , E_1(R^*))$ between exact couples. By this we mean a pair of morphisms $\delta : A_1(T^*) \to A_1(R^*)$, $\delta : E_1(T^*) \to E_1(R^*)$ intertwining the maps $i,j,k$ of the exact couples. More precisely we will define maps $\delta : A_1^{p,q}(T^*) \to A_1^{p,q+1}(R^*)$, $\delta : E_1^{p,q}(T^*) \to E_1^{p,q+1}(R^*)$ as follows. The short exact sequence $0 \to R^* \to S^* \to T^* \to 0$ yields corresponding short exact sequences $0 \to F^p R^* \to F^p S^* \to F^p T^* \to 0$ on the filtrations and thus we obtain boundary maps $H^n( F^p T^*) \to H^{n+1}( F^p R^*)$. This defines the map $\delta : A_1^{p,q}(T^*) \to A_1^{p,q+1}(R^*)$. By the Nine lemma we get exact sequences $0 \to F^p R^* / F^{p+1} R^* \to F^p S^* / F^{p+1} S^* \to F^p T^* / F^{p+1} T^* \to 0$ and thus boundary maps $H^n( F^p T^* / F^{p+1} T^* ) \to H^{n+1}( F^p R^* / F^{p+1} R^* )$ which we take as the definition of $\delta : E_1^{p,q}(T^*) \to E_1^{p,q+1}(T^*)$. One needs to show that the maps $\delta$ so defined intertwine\footnote{strictly speaking the maps $\delta$ only intertwine $i,j,k$ up to certain irrelevant sign factors} the maps $i,j,k$ and then we have a morphism between exact couples. When one passes to the derived exact couples the $\delta$ maps induce corresponding maps on the derived exact couples. Thus we get maps $\delta_r : E_r^{p,q}(T^*) \to E_r^{p,q+1}(R^*)$ which are the maps in the statement of the theorem. By definition of the filtrations $F^p H^n(R^*)$, $F^p H^n(T^*)$ we see that the boundary map $H^n(T^*) \to H^{n+1}(R^*)$ sends $F^p H^n(T^*)$ to $F^p H^{n+1}(R^*)$. Upon identification of $H^n(R^*)$ with $R^n(G \circ F)(A)$ and $H^n(T^*)$ with $R^n(G \circ F)(C)$ we see as claimed that the natural boundary map $R^n(G \circ F)(C) \to R^{n+1}(G \circ F)(A)$ sends $F^p R^n(G \circ F)(C)$ to $F^p R^{n+1}(G \circ F)(A)$ and that the induced maps $E_\infty^{p,q}(T^*) \to E_\infty^{p,q+1}(R^*)$ coincide with $\delta_\infty$.\\

To finish the proof we must show that $\delta_2$ has the expected form. First note that $E_1^{p,q}(R^*)$ can be identified with $G( H^q( I^{p,*} ) )$ and similarly for $T^*$, while the boundary maps $E_1^{p,q}(T^*) \to E_1^{p,q+1}(R^*)$ are easily seen to arise from the boundary maps $H^{q}( K^{p,*} ) \to H^{q+1}( I^{p,*} )$. To pass from $E_1$ to $E_2$ one then takes cohomology with respect to the differential induced by the maps $I^{*,p} \to I^{*,p+1}$. Since $I^{*,*}$ is a Cartan-Eilenberg resolution we have that $\{ H^q( I^{p,*}) \}_{p \ge 0}$ is an injective resolution of $H^q (F(M^*)) = R^qF(A)$. On applying $G$ and taking cohomology in the $p$ direction we get $E_2^{p,q}(R^*) = R^pG ( R^qF(A))$. As usual a similar statement holds for $C$. Now observe that $H^{q}( K^{p,*} ) \to H^{q+1}( I^{p,*} )$ is a map between injective resolutions of $H^q( F(P^*)) = R^q F(C) $ and $H^{q+1}( F(M^*)) = R^{q+1}F(A)$ commuting with the boundary map $R^q F(C) \to R^{q+1}F(A)$ so on applying $G$ and taking cohomology in the $p$ direction we see that the induced map $E_2^{p,q}(C) \to E_2^{p,q+1}(A)$ is indeed the map $R^pG (R^qF(C)) \to R^pG (R^{q+1}F(A))$ induced by the boundary $R^qF(C) \to R^{q+1} F(A)$.
\end{proof}


\section{Application}\label{app}

The main application of Theorem \ref{main} we have in mind is to the Leray spectral sequence for sheaf cohomology. Let $X,Y$ be topological spaces and $f : X \to Y$ a continuous map. We take $\mathcal{A},\mathcal{B}$ to be the categories of sheaves of abelian groups on $X$ and $Y$ respectively and $\mathcal{C}$ the category of abelian groups. We take the functor $F : \mathcal{A} \to \mathcal{B}$ to be the push-forward $F = f_*$ under $f$ and $G : \mathcal{B} \to \mathcal{C}$ to be the global sections functor $G = \Gamma$. Note that $G \circ F$ is the global sections functor for sheaves on $X$.\\

To derive the Leray spectral sequence from the Grothendieck spectral sequence one only needs to note that the category of sheaves of abelian groups on a topological space has enough injectives \cite[Lemma 1.1.13]{bry} and that the push-forward functor $F = f_*$ actually sends injectives to injectives \cite[Lemma 1.6.3]{bry}. For any sheaf $A$ on $X$ we thus obtain the Leray spectral sequence, a spectral sequence $\{ E_r^{p,q} , d_r \}$ which abuts to the sheaf cohomology $H^*(X,A)$ and such that $E_2^{p,q}(A) = H^p(Y, R^qf_* A)$. As with the Grothendieck spectral sequence there is a natural filtration
\begin{eqnarray*}
0 = F^{n+1,n}(A) \subseteq F^{n,n}(A) \subseteq \cdots \subseteq F^{1,n}(A) \subseteq F^{0,n}(A) = H^n(X,A)
\end{eqnarray*}
related to the spectral sequence by 
\begin{equation*}
E_\infty^{p,q}(A) \simeq F^{p,p+q}(A)/F^{p+1,p+q}(A).
\end{equation*}
Theorem \ref{main} translated to this situation becomes
\begin{theorem}\label{lss}
Let $X,Y$ be topological spaces and $f : X \to Y$ a continuous map. Let $0 \to A \to B \to C \to 0$ be a short exact sequence of sheaves. There are morphisms $\delta_r : E_r^{p,q}(C) \to E_r^{p,q+1}(A)$ for $r \ge 2$ between the Leray spectral sequences for $C$ and $A$ with the following properties:
\begin{itemize}
\item{$\delta_r$ commutes with the differentials $d_r$ and the induced map at the $(r+1)$-stage is $\delta_{r+1}$.}
\item{$\delta_2 : H^p( Y , R^q f_*C) \to H^p( Y , R^{q+1}f_*A)$ is the map induced by the boundary morphism $R^qf_*C \to R^{q+1}f_*A$ in the long exact sequence of higher direct image functors of $f$ associated to $0 \to A \to B \to C \to 0$.}
\item{The boundaries $H^n(X,C) \to H^{n+1}(X,A)$ for the long exact sequence of sheaf cohomology send $F^{p,n}(C)$ to $F^{p,n+1}(A)$ and thus induce maps $E_\infty^{p,q}(C) \to E_\infty^{p,q+1}(A)$. These maps coincide with $\delta_\infty$ where $\delta_\infty$ denotes the limit of the $\delta_r$.}
\end{itemize}
\end{theorem}

Let us now consider a specific application of this result we have in mind. Henceforth we assume that the topological spaces $X,Y$ are paracompact. Furthermore assume that every subspace of $X$ is paracompact. This is the case for instance if $X$ is a metric space \cite[Theorem 5.13]{eng}. We use the notation $\underline{\mathbb{C}},\underline{\mathbb{C}}^*$ to denote the sheaves of continuous functions with values in $\mathbb{C},\mathbb{C}^*$, where $\mathbb{C}^*$ is the non-zero complex numbers.

\begin{theorem}\label{cz}
Let $f : X \to Y$ be a map between spaces $X,Y$, let $E_r^{p,q}(\mathbb{Z}),E_r^{p,q}(\underline{\mathbb{C}}^*)$ be the Leray spectral sequences associated to the sheaves $\mathbb{Z},\underline{\mathbb{C}}^*$ and let $F^{p,n}(\mathbb{Z}), \linebreak F^{p,n}(\underline{\mathbb{C}}^*)$ be the associated filtrations on $H^n(X,\mathbb{Z})$,$H^n(X,\underline{\mathbb{C}}^*)$. Then
\begin{itemize}
\item{The coboundary $\delta : H^n(X,\underline{\mathbb{C}}^*) \to H^{n+1}(X,\mathbb{Z})$ restricts to morphisms $\delta : F^{p,p+q}(\underline{\mathbb{C}}^*) \to F^{p,p+q+1}(\mathbb{Z})$ which are isomorphisms whenever $p+q \ge 1$ and surjective for $p=q=0$.}
\item{The induced quotient maps $E_\infty^{p,q}(\underline{\mathbb{C}}^*) \to E_\infty^{p,q+1}(\mathbb{Z})$ are isomorphisms for $q \ge 1$ and surjections for $q = 0$.}
\end{itemize}
\end{theorem}
\begin{proof}
We apply Theorem \ref{lss} to the exponential sequence $0 \to \mathbb{Z} \to \underline{\mathbb{C}} \to \underline{\mathbb{C}}^* \to 0$. Thus the coboundary $\delta : H^n(X,\underline{\mathbb{C}}^*) \to H^{n+1}(X,\mathbb{Z})$ restricts to morphisms $\delta : F^{p,p+q}(\underline{\mathbb{C}}^*) \to F^{p,p+q+1}(\mathbb{Z})$ and induces quotient maps $\delta_\infty : E_\infty^{p,q}(\underline{\mathbb{C}}^*) \to E_\infty^{p,q+1}(\mathbb{Z})$. Next we observe that $E_2^{p,q}(\mathbb{Z}) = H^p(Y , R^q f_* \mathbb{Z})$ and $E_2^{p,q}(\underline{\mathbb{C}}^*) = H^p(Y , R^qf_* \underline{\mathbb{C}}^*)$. The natural maps $R^qf_* \underline{\mathbb{C}}^* \to R^{q+1}f_* \mathbb{Z}$ in the long exact sequence of higher direct image sheaves are isomorphisms for $q \ge 1$, since $R^q f_* \underline{\mathbb{C}} = 0$ for $q \ge 1$. We therefore have that $\delta_2 : E_2^{p,q}(\underline{\mathbb{C}}^*) \to E_2^{p,q+1}(\mathbb{Z})$ is an isomorphism for $q \ge 1$.\\

We will now show that for $q \ge 1$ the maps $\delta_\infty : E_\infty^{p,q}(\underline{\mathbb{C}}^*) \to E_\infty^{p,q+1}(\mathbb{Z})$ are injective. In fact we start by showing that the maps $\delta_3 : E_3^{p,q}(\underline{\mathbb{C}}^*) \to E_3^{p,q+1}(\mathbb{Z})$ are injective for $q \ge 1$. Let $x \in E_3^{p,q}(\underline{\mathbb{C}}^*)$ be such that $\delta_3(x) = 0$. Choose a representative $\tilde{x} \in E_2^{p,q}(\underline{\mathbb{C}}^*)$ for $x$. Then $\delta_3(x) = 0$ means that $\delta_2(\tilde{x}) = d_2 \tilde{y}$ for some $\tilde{y} \in E_2^{p-2,q+2}(\mathbb{Z})$. We can then find $\tilde{z} \in E_2^{p-2,q+1}(\underline{\mathbb{C}}^*)$ so that $\tilde{y} = \delta_2(\tilde{z})$ and thus $\delta_2(\tilde{x}) = d_2 \tilde{y} = d_2 \delta_2 \tilde{z} = \delta_2 d_2 \tilde{z}$. By injectivity of $\delta_2$ we have $\tilde{x} = d_2 \tilde{z}$ and thus $x = 0$ proving injectivity of $\delta_3$. Proceeding by induction we find that $\delta_r : E_r^{p,q}(\underline{\mathbb{C}}^*) \to E_r^{p,q+1}(\mathbb{Z})$ is injective for all $r \ge 2$ and $q \ge 1$.\\

Next we show that the maps $\delta_\infty : E_\infty^{p,q}(\underline{\mathbb{C}}^*) \to E_\infty^{p,q+1}(\mathbb{Z})$ are surjective for all $p,q$. To begin let $x \in E_\infty^{0,q+1}(\mathbb{Z})$ and lift $x$ to a class $\tilde{x} \in F^{0,q+1}(\mathbb{Z}) = H^{q+1}(X,\mathbb{Z})$. Since $\delta : H^q(X , \underline{\mathbb{C}}^* ) \to H^{q+1}(X,\mathbb{Z})$ is surjective we can find $\tilde{y} \in H^q(X , \underline{\mathbb{C}}^* )$ so that $\delta(\tilde{y}) = \tilde{x}$. Projecting $\tilde{y}$ to a class $y \in E_\infty^{0,q}(\underline{\mathbb{C}}^*)$ we find that $\delta_\infty(y) = x$ proving surjectivity in the $p=0$ case. Now we proceed by induction on $p$, so assume that $\delta_\infty : E_\infty^{k,q}(\underline{\mathbb{C}}^*) \to E_\infty^{k,q+1}(\mathbb{Z})$ is surjective for all $q$ and all $k \le p-1$, where now $p > 0$. Given $x \in E_\infty^{p,q+1}(\mathbb{Z})$ lift $x$ to a class $\tilde{x} \in F^{p,p+q+1}(\mathbb{Z}) \subseteq H^{p+q+1}(X,\mathbb{Z})$. Then since $\delta : H^{p+q}(X,\underline{\mathbb{C}}^*) \to H^{p+q+1}(X,\mathbb{Z})$ is surjective we can find $\tilde{y} \in H^{p+q}(X,\underline{\mathbb{C}}^*)$ such that $\delta \tilde{y} = \tilde{x}$. To proceed we need to argue that $\tilde{y} \in F^{p,p+q}(\underline{\mathbb{C}}^*)$. Notice that $\tilde{y} \in F^{0,p+q}(\underline{\mathbb{C}}^*) = H^{p+q}(X,\underline{\mathbb{C}}^*)$. Using the fact that $\delta_\infty : E_\infty^{0,p+q}(\underline{\mathbb{C}}^*) \to E_\infty^{0,p+q+1}(\mathbb{Z})$ is injective (since $p+q > 0$) and that the image of $\tilde{x}$ in $E_\infty^{0,p+q+1}(\mathbb{Z})$ is zero we see that correspondingly the image of $\tilde{y}$ in $E_\infty^{0,p+q}(\underline{\mathbb{C}}^*)$ is zero and thus $\tilde{y} \in F^{1,p+q}(\underline{\mathbb{C}}^*)$. Continuing in this fashion using the injectivity of the $\delta_\infty : E_\infty^{a,b}(\underline{\mathbb{C}}^*) \to E_\infty^{a,b+1}(\mathbb{Z})$ for $b \ge 1$ we get that $\tilde{y} \in F^{2,p+q}(\underline{\mathbb{C}}^*)$, $\tilde{y} \in F^{3,p+q}(\underline{\mathbb{C}}^*), \dots$ and eventually get that $\tilde{y} \in F^{p,p+q}(\underline{\mathbb{C}}^*)$. Projecting $\tilde{y}$ to $y \in E_\infty^{p,q}(\underline{\mathbb{C}}^*)$ we immediately see that $\delta_\infty(y) = x$ proving surjectivity of $\delta_\infty$.\\

Using the fact that the $\delta_\infty : E_\infty^{p,q}(\underline{\mathbb{C}}^*) \to E_\infty^{p,q+1}(\mathbb{Z})$ are surjective for all $p,q$ we easily see that the maps $\delta : F^{p,p+q}(\underline{\mathbb{C}}^*) \to F^{p,p+q+1}(\mathbb{Z})$ surject onto the quotients $F^{p,p+q+1}(\mathbb{Z})/F^{p+q+1,p+q+1}(\mathbb{Z})$ for all $p,q$. Given this we can show the maps $\delta_\infty$ are surjective by showing that every class $x \in F^{p+q+1,p+q+1}(\mathbb{Z})$ has the form $x = \delta(y)$ for some $y \in F^{p+q,p+q}(\underline{\mathbb{C}}^*)$. Since $\delta : H^{p+q}(X,\underline{\mathbb{C}}^*) \to H^{p+q+1}(X,\mathbb{Z})$ is surjective we can find $y \in H^{p+q}(X,\underline{\mathbb{C}}^*)$ so that $x = \delta(y)$. If we project $y$ to a class in $E_\infty^{0,p+q}(\underline{\mathbb{C}}^*)$ and use injectivity of $\delta_\infty : E_\infty^{0,p+q}(\underline{\mathbb{C}}^*) \to E_\infty^{0,p+q+1}(\mathbb{Z})$ (if $p+q \ge 1$) we find that the projection of $y$ to $E_\infty^{0,p+q}(\underline{\mathbb{C}}^*)$ is zero so that $y \in F^{p+q-1,p+q}(\underline{\mathbb{C}}^*)$. Continuing on in this fashion we work our way down the filtration and ultimately find that $y \in F^{p+q,p+q}(\underline{\mathbb{C}}^*)$ as claimed.\\

Finally to see that the maps $\delta : F^{p,p+q}(\underline{\mathbb{C}}^*) \to F^{p,p+q+1}(\mathbb{Z})$ are injective whenever $p+q \ge 1$ we just need to note that $\delta : H^{p+q}(X,\underline{\mathbb{C}}^*) \to H^{p+q+1}(X,\mathbb{Z})$ is injective, indeed an isomorphism whenever $p+q \ge 1$.
\end{proof}

\begin{remark}
Observe that in the proof of Theorem \ref{cz} the only property of the exponential sequence $0 \to \mathbb{Z} \to \underline{\mathbb{C}} \to \underline{\mathbb{C}}^* \to 0$ that is used is that the sheaf $\underline{\mathbb{C}}$ on $X$ has the property that its restriction to any open subset of $X$ is acyclic. Thus the result carries over to many other short exact sequences of sheaves. For example given $\epsilon \in H^1(X,\mathbb{Z}_2)$, tensoring by the corresponding local system $\mathbb{Z}_\epsilon$ yields an exact sequence $0 \to \mathbb{Z}_\epsilon \to \underline{\mathbb{C}}_\epsilon \to \underline{\mathbb{C}}^*_\epsilon \to 0$ that also satisfies this condition.
\end{remark}

Theorem \ref{cz} has applications to topological T-duality. Suppose that $f : X \to Y$ is a principal $T^n$-bundle over $Y$ where $T^n = \mathbb{R}^n/\mathbb{Z}^n$ is the $n$-torus and suppose $\mathcal{G}$ is a bundle gerbe on $X$ \cite{mm},\cite{stev} which up to stable isomorphism is classified by its Dixmier-Douady class $h = [\mathcal{G}] \in H^2(X,\underline{\mathbb{C}}^*) = H^3(X,\mathbb{Z})$. Bunke, Rumpf and Schick \cite{brs} give a definition of topological $T$-duality for the data $(f : X \to Y, \mathcal{G})$ building upon the definition of topological $T$-duality first introduced in \cite{bem}. In \cite{brs} it is established that $(X,\mathcal{G})$ admits a T-dual (in the sense of \cite{brs}) if and only if $h = [\mathcal{G}] \in F^{2,3}(\mathbb{Z})$. From Theorem \ref{cz} we see that this is equivalent to $h \in F^{2,2}(\underline{\mathbb{C}}^*)$ if we regard $h$ as an element of $H^2(X,\underline{\mathbb{C}}^*)$. Next we observe that there is a natural surjection $E_2^{2,0}(\underline{\mathbb{C}}^*) \to F^{2,2}(\underline{\mathbb{C}}^*)$. Thus the T-dualizable classes on $X$ are precisely the image of $E_2^{2,0}(\underline{\mathbb{C}}^*) = H^2(Y, f_*(\underline{\mathbb{C}}^*))$ under the natural map $H^2(Y, f_*(\underline{\mathbb{C}}^*)) \to H^2(X,\underline{\mathbb{C}}^*)$. We claim that by describing T-dualizable gerbes in terms of classes in $H^2(Y, f_*(\underline{\mathbb{C}}^*))$ the proof of certain existence results in T-duality greatly simplify, a claim that we intend to show in \cite{bar1}.


\bibliographystyle{amsplain}

\end{document}